\documentclass[12pt]{article}

\usepackage{amsmath, epsfig, cite}
\usepackage{amssymb}
\usepackage{amsfonts}
\usepackage{latexsym}
\usepackage{amsthm}

\newtheorem{thm}{Theorem}[section]

\newtheorem{cor}[thm]{Corollary}
\newtheorem{conj}[thm]{Conjecture}
\newtheorem{lem}[thm]{Lemma}

\newcommand{\binq}[2]{\genfrac{[}{]}{0mm}{0}{#1}{#2}}
\newcommand{\tbnq}[2]{\genfrac{[}{]}{0mm}{1}{#1}{#2}}
\numberwithin{equation}{section}

\renewcommand{\thefootnote}{*}

\begin{document}

\begin{center}
{\large\bf Some $q$-supercongruences modulo the fourth power of a
cyclotomic polynomial \footnote{This work is supported by the
National Natural Science Foundation of China (No. 11661032).}}
\end{center}

\renewcommand{\thefootnote}{$\dagger$}

\vskip 2mm \centerline{Chuanan Wei}
\begin{center}
{School of Biomedical Information and Engineering,\\ Hainan Medical University, Haikou 571199, China\\
{\tt weichuanan78@163.com  } }
\end{center}


\vskip 0.7cm \noindent{\bf Abstract.} In terms of the creative
microscoping method recently introduced by Guo and Zudilin and the
Chinese remainder theorem for coprime polynomials, we establish a
$q$-supercongruence with two parameters modulo $[n]\Phi_n(q)^3$.  Here $[n]=(1-q^n)/(1-q)$ and
$\Phi_n(q)$ is the $n$-th cyclotomic polynomial in $q$.
In particular, we confirm a recent conjecture of Guo and give a complete
$q$-analogue of Long's supercongruence. The latter is also a generalization
of a recent $q$-supercongruence obtained by Guo and Schlosser.



\section{Introduction}
For a complex variable $x$, define the shifted-factorial to be
\[(x)_{0}=1\quad \text{and}\quad (x)_{n}
=x(x+1)\cdots(x+n-1)\quad \text{when}\quad n\in\mathbb{N}.\]
 In his second letter to Hardy on February 27, 1913, Ramanujan mentioned the
identity
\begin{equation}\label{ramanujan}
\sum_{k=0}^{\infty}(-1)^k(4k+1)\frac{(\frac{1}{2})_k^5}{k!^5}=\frac{2}{\Gamma(\frac{3}{4})^4},
\end{equation}
where $\Gamma(x)$ is the Gamma function. In 1997, Van Hamme
\cite[(A.2)]{Hamme} conjectured that \eqref{ramanujan} possesses the
nice $p$-adic analogue:
\begin{equation}\label{hamme-b}
\sum_{k=0}^{(p-1)/2}(-1)^k(4k+1)\frac{(\frac{1}{2})_k^5}{k!^5}\equiv
\begin{cases} \displaystyle -\frac{p}{\Gamma_p(\frac{3}{4})^4}  \pmod{p^3}, &\text{if $p\equiv 1\pmod 4$,}\\[10pt]
 0\pmod{p^3}, &\text{if $p\equiv 3\pmod 4$.}
\end{cases}
\end{equation}
Here and throughout the paper, $p$ always denotes an odd prime and
$\Gamma_p(x)$ is the $p$-adic Gamma function. The supercongruence
\eqref{hamme-b} was later proved by McCarthy and Osburn
\cite{McCarthy}. In 2015, Swisher \cite{Swisher} showed that
\eqref{hamme-b} also holds modulo $p^5$ for $p\equiv 1\pmod 4$ and
$p>5$. Recently, Liu \cite{Liu} found another partial generalization
of it: for $p\equiv 3\pmod 4$ and $p>3$,
\begin{equation}\label{Liu}
\sum_{k=0}^{(p-1)/2}(-1)^k(4k+1)\frac{(\frac{1}{2})_k^5}{k!^5}\equiv
-\frac{p^3}{16}\Gamma_p\bigg(\frac{1}{4}\bigg)^4\pmod{p^4}.
\end{equation}

It is known that some of the truncated hypergeometric series are
related to the number of rational points on certain algebraic
varieties over finite fields and further to coefficients of modular
forms. For example, on the basis of the result of Ahlgren and Ono in
\cite{Ahlgren-Ono}, Kilbourn \cite{Kilbourn} proved Van Hamme¡¯s
(M.2) supercongruence:
\begin{equation}\label{Hamme-c}
\sum_{k=0}^{(p-1)/2}\frac{(\frac{1}{2})_k^4}{k!^4}\equiv
a_p\pmod{p^{3}},
\end{equation}
where $a_p$ is the $p$-th coefficient of a weight 4 modular form
\begin{equation*}
\eta(2z)^4\eta(4z)^4:=q\prod_{n=1}^{\infty}(1-q^{2n})^4(1-q^{4n})^4,\quad
q=e^{2i\pi z}.
\end{equation*}
 In 2011, Long \cite{LR} obtained the following two supercongruences:
\begin{align}
\sum_{k=0}^{(p-1)/2}(4k+1)\frac{(\frac{1}{2})_k^4}{k!^4} &\equiv p\pmod{p^{4}},  \label{long-a} \\[5pt]
\sum_{k=0}^{(p-1)/2}(4k+1)\frac{(\frac{1}{2})_k^6}{k!^6} &\equiv
p\sum_{k=0}^{(p-1)/2}\frac{(\frac{1}{2})_k^4}{k!^4}\pmod{p^4},  \label{long-c}
\end{align}
where $p>3$. According to
\eqref{Hamme-c}, the supercongruence \eqref{long-c} can be written
as
\begin{equation*}
\sum_{k=0}^{(p-1)/2}(4k+1)\frac{(\frac{1}{2})_k^6}{k!^6} \equiv
pa_p\pmod{p^4}\quad\text{for}\ p>3.
\end{equation*}

 For two complex numbers $x$ and $q$, define the $q$-shifted factorial to be
 \begin{equation*}
(x;q)_{0}=1\quad\text{and}\quad
(x;q)_n=(1-x)(1-xq)\cdots(1-xq^{n-1})\quad \text{when}\quad
n\in\mathbb{N}.
 \end{equation*}
For shortening many of the formulas in this paper, we adopt the
notation
\begin{equation*}
(x_1,x_2,\dots,x_r;q)_{n}=(x_1;q)_{n}(x_2;q)_{n}\cdots(x_r;q)_{n}.
 \end{equation*}
Following Gasper and Rahman \cite{Gasper}, define the basic
hypergeometric series by
$$
_{r}\phi_{s}\left[\begin{array}{c}
a_1,a_2,\ldots,a_{r}\\
b_1,b_2,\ldots,b_{s}
\end{array};q,\, z
\right] =\sum_{k=0}^{\infty}\frac{(a_1,a_2,\ldots, a_{r};q)_k}
{(q,b_1,b_2,\ldots,b_{s};q)_k}\bigg\{(-1)^kq^{\binom{k}{2}}\bigg\}^{1+s-r}z^k.
$$
Then the $q$-Chu--Vandermonde sum (cf. \cite[Appendix
(II.6)]{Gasper}), a terminating $q$-analogue of Whipple's $_3F_2$
sum (cf. \cite[Appendix (II.19)]{Gasper}) and Watson's $_8\phi_7$ transformation (cf. \cite[Appendix
(III.18)]{Gasper}) can be expressed as follows:
\begin{align}
&\qquad\qquad\qquad\qquad
 _{2}\phi_{1}\!\left[\begin{array}{cccccccc}
 q^{-n},  &b \\ &c
\end{array};q,\, \frac{cq^n}{b} \right]
=\frac{(c/b;q)_{n}}{(c;q)_{n}}, \label{eq:q-Chu-Vandermonde}
\\[10pt]
&\quad
 _{4}\phi_{3}\!\left[\begin{array}{cccccccc}
 q^{-n},  &q^{1+n}, &b, & -b \\
 &c,  &b^2q/c, &-q
\end{array};q,\, q \right]
=q^{\binom{1+n}{2}}\frac{(b^2q^{1-n}/c, cq^{-n};q^2)_{n}}
 {(b^2q/c, c;q)_{n}}, \label{eq:whipple-a}
\\[10pt]
& _{8}\phi_{7}\!\left[\begin{array}{cccccccc}
a,& qa^{\frac{1}{2}},& -qa^{\frac{1}{2}}, & b,    & c,    & d,    & e,    & q^{-n} \\
  & a^{\frac{1}{2}}, & -a^{\frac{1}{2}},  & aq/b, & aq/c, & aq/d, & aq/e, & aq^{n+1}
\end{array};q,\, \frac{a^2q^{n+2}}{bcde}
\right] \notag\\[5pt]
&\quad =\frac{(aq, aq/de;q)_{n}} {(aq/d, aq/e;q)_{n}}
\,{}_{4}\phi_{3}\!\left[\begin{array}{c}
aq/bc,\ d,\ e,\ q^{-n} \\
aq/b,\, aq/c,\, deq^{-n}/a
\end{array};q,\, q
\right]. \label{eq:watson}
\end{align}

Recently, Guo \cite{Guo-new} proved that, for any
positive odd integer $n\equiv3\pmod{4}$,
\begin{equation}
\sum_{k=0}^{(n-1)/2}\frac{(q;q^2)_k^2(q^2;q^4)_k}{(q^2;q^2)_k^2(q^4;q^4)_k}q^{2k}\equiv
[n]\frac{(q^3;q^4)_{(n-1)/2}}{(q^5;q^4)_{(n-1)/2}}\pmod{\Phi_n(q)^3}
\label{guo-a}
\end{equation}
and proposed the following conjecture: for any positive odd integer
$n\equiv3\pmod{4}$,
\begin{align}
\sum_{k=0}^{M}(-1)^k[4k+1]\frac{(q;q^2)_k^4(q^2;q^4)_k}{(q^2;q^2)_k^4(q^4;q^4)_k}q^k
\equiv[n]^2q^{(1+n)/2}\frac{(q^3;q^4)_{(n-1)/2}}{(q^5;q^4)_{(n-1)/2}}\pmod{[n]\Phi_n(q)^3},
\label{guo-conjecture}
\end{align}
which is a $q$-analogue of \eqref{Liu}. Here and throughout the
paper, $M$ is always equal to $(n-1)/2$ or $(n-1)$. Some different
works can be stated as follows. Guo and Zudilin \cite[Theorem
4.2]{GuoZu} found the formula with two parameters: for a positive
odd integer $n$,
\begin{align}
&\sum_{k=0}^{M}[4k+1]\frac{(aq,q/a,q/c,q;q^2)_k}{(q^2/a,aq^2,cq^2,q^2;q^2)_k}c^k
\notag\\[5pt]
&\:\equiv
[n]\frac{(c/q)^{(n-1)/2}(q^2/c;q^2)_{(n-1)/2}}{(cq^2;q^2)_{(n-1)/2}}\pmod{[n](1-aq^n)(a-q^n)}.
\label{guo-b}
\end{align}
 Guo and Wang \cite{Guo-Wang} achieved a $q$-analogue of
\eqref{long-a}: for any positive odd integer $n$,
\begin{equation}
\sum_{k=0}^{M}[4k+1]\frac{(q;q^2)_k^4}{(q^2;q^2)_k^4}\equiv
[n]q^{(1-n)/2}+[n]^3q^{(1-n)/2}\frac{(n^2-1)(1-q)^2}{24}\pmod{[n]\Phi_n(q)^3}.
\label{guo-c}
\end{equation}
Guo and Schlosser \cite[Theorems 2.1 and 2.2]{GS}) proved that, for
a positive odd integer $n$,
\begin{align}
&\sum_{k=0}^{M}[4k+1]\frac{(q;q^2)_k^6}{(q^2;q^2)_k^6}q^k
\equiv[n]q^{(1-n)/2}\sum_{k=0}^{(n-1)/2}\frac{(q;q^2)_k^4}{(q^2;q^2)_k^4}q^{2k}\hspace*{-1.5mm}\pmod{[n]\Phi_n(q)^2},
\label{guo-d}\\[5pt]
&\sum_{k=0}^{M}(-1)^k[4k+1]\frac{(q;q^2)_k^5}{(q^2;q^2)_k^5}q^{k^2+k}
\equiv[n]q^{(1-n)/2}\sum_{k=0}^{(n-1)/2}\frac{(q;q^2)_k^3}{(q^2;q^2)_k^3}q^{2k}\hspace*{-1.5mm}\pmod{[n]\Phi_n(q)^2}.
\label{guo-e}
\end{align}
 The
$q$-supercongruence \eqref{guo-d} is a $q$-analogue of
\eqref{long-c}, where the modulo $p^4$ condition is replaced by the
weaker condition modulo $p^3$. An indeed $q$-analogue of \eqref{hamme-b} (cf. \cite{Guo-a2} and \cite {WY})
can be expressed as follows:  for any positive odd integer $n$,
\begin{align}\label{guo-f}
&\sum_{k=0}^{M}(-1)^k[4k+1]\frac{(q;q^2)_k^4(q^2;q^4)_k}{(q^2;q^2)_k^4(q^4;q^4)_k}q^k
\notag\\[5pt]
&\equiv
\begin{cases} \displaystyle [n]\frac{(q^2;q^4)_{(n-1)/4}^2}{(q^4;q^4)_{(n-1)/4}^2}  \pmod{[n]\Phi_n(q)^2}, &\text{if $n\equiv 1\pmod 4$,}\\[15pt]
 0  \pmod{[n]\Phi_n(q)^2}, &\text{if $n\equiv 3\pmod 4$.}
\end{cases}
\end{align}
We point out that more $q$-analogues of supercongruences can be
found in \cite{GuoAdcance,Guo-jmaa,GS20c,LP,NP,Tauraso,WY-a,Zu19}
with various techniques.

Inspired by these work just mentioned, we shall establish the
following theorem.

\begin{thm}\label{thm-a}
Let $n$ be a positive odd integer. Then
\begin{align}
&\sum_{k=0}^{M}[4k+1]\frac{(q;q^2)_k^4(cq,dq;q^2)_k}{(q^2;q^2)_k^4(q^2/c,q^2/d;q^2)_k}\bigg(\frac{q}{cd}\bigg)^k
\notag\\[5pt]\:
&\:\equiv\bigg\{[n]q^{(1-n)/2}+[n]^3q^{(1-n)/2}\frac{(n^2-1)(1-q)^2}{24}\bigg\}
\notag\\[5pt]
&\quad\times\sum_{k=0}^{(n-1)/2}\frac{(q;q^2)_k^3(q/cd;q^2)_k}{(q^2;q^2)_k^2(q^2/c,q^2/d;q^2)_k}q^{2k}\pmod{[n]\Phi_n(q)^3}.  \label{eq:wei}
\end{align}
\end{thm}

When $cd=q$, the $q$-supercognruence \eqref{eq:wei} reduces to \eqref{guo-c}. When the parameters $c$
and $d$ are further specified,  we can confirm Guo's conjecture \eqref{guo-conjecture}, a $q$-analogue of \eqref{long-c}, five new
$q$-analogues of \eqref{long-a}, and some other conclusions from
this theorem.

The rest of the paper is arranged as follows. We shall display
several concrete
 $q$-supercongruences from Theorem \ref{thm-a} in Section 2. Via the
 creative microscoping method, a $q$-supercongruence with four parameters modulo $\Phi_n(q)(1-aq^n)(a-q^n)$, which includes \eqref{guo-b}
and \eqref{guo-d}--\eqref{guo-f} as special cases, will be derived
in Section 3. Then we utilize it and the Chinese remainder theorem
for coprime polynomials to deduce a $q$-supercongruence with three
parameters modulo $\Phi_n(q)^2(1-aq^n)(a-q^n)$ and prove Theorem
\ref{thm-a} in Section 4.
\section{Concrete $q$-supergruences from Theorem \ref{thm-a}}
Nine $q$-supercongruences
 modulo $[n]\Phi_n(q)^3$ from Theorem
\ref{thm-a} will be laid out. Above all, we give the following lemma.

\begin{lem}\label{lem-zero}
Let $n$ be a positive odd integer. Then
\begin{align*}
[n]^2\frac{(q^3;q^4)_{(n-1)/2}}{(q^5;q^4)_{(n-1)/2}}\equiv0\pmod{[n]}.
\end{align*}
\end{lem}

\begin{proof}
For two nonnegative integer $s,t$ with $s\leq t$, it is well known
that the $q$-binomial coefficient $\tbnq{t}{s}$ is a polynomial in
$q$ and
\begin{align}\label{polynomial}
\frac{(q;q^2)_t}{(q^2;q^2)_t}=\frac{1}{(-q;q)_t^2}\binq{2t}{t}.
\end{align}
By specifying the parameters in \eqref{eq:whipple-a}, Guo
\cite{Guo-new} discovered the identity
\begin{align*}
& _{4}\phi_{3}\!\left[\begin{array}{cccccccc}
 q^{1-n},  &q^{1+n}, &q,  &-q \\
 &q^{2+n}, &q^{2-n},  &-q^2
\end{array};q^2,\, q^2 \right]
=[n]\frac{(q^3;q^4)_{(n-1)/2}}{(q^5;q^4)_{(n-1)/2}}.
\end{align*}
For proving Lemma \ref{lem-zero}, it is
sufficient to show that
\begin{align}
[n]\frac{(q^{1-n},q^{1+n};q^2)_{k}}{(q^{2-n},q^{2+n};q^2)_{k}}\equiv0\pmod{[n]},
\label{eq:wei-ab}
\end{align}
where $0\leq k\leq(n-1)/2$. Through \eqref{eq:q-Chu-Vandermonde}, we
have
\begin{align*}
[n]\frac{(q^{1-n},q^{1+n};q^2)_{k}}{(q^{2-n},q^{2+n};q^2)_{k}}
&=q^{-k}[n]^2\frac{(q;q^2)_{(n-1)/2-k}}{(q^2;q^2)_{(n-1)/2-k}}\frac{(q^2;q^2)_{(n-1)/2+k}}{(q^3;q^2)_{(n-1)/2+k}}
\notag\\[5pt]
&=\frac{(q;q^2)_{(n-1)/2-k}}{(q^2;q^2)_{(n-1)/2-k}}
\notag\\[5pt]
&\quad\times\sum_{j=0}^{(n-1)/2+k}(-1)^jq^{j^2+j-k}\binq{(n-1)/2+k}{j}_{q^2}\frac{[n]^2}{[1+2j]}.
\end{align*}
Thus we verify the correctness of \eqref{eq:wei-ab}. This finishes
the proof of Lemma \ref{lem-zero}.
\end{proof}

 It is easy to understand that the factor $(q^2/c,q^2/d;q^2)_M$
 in the denominator of the left-hand side of  \eqref{eq:wei} is
 relatively prime to $\Phi_n(q)$ as $c\to1,d\to-1$ (some similar discussion will be omitted in the rest of the paper).
Choosing $c=1,d=-1$ in Theorem \ref{thm-a}, we obtain
\begin{align*}
\sum_{k=0}^{M}(-1)^k[4k+1]\frac{(q;q^2)_k^4(q^2;q^4)_k}{(q^2;q^2)_k^4(q^4;q^4)_k}q^k
&\equiv[n]q^{(1-n)/2}\bigg\{1+[n]^2\frac{(n^2-1)(1-q)^2}{24}\bigg\}
\end{align*}
\begin{align*}
&\quad\times\sum_{k=0}^{(n-1)/2}\frac{(q;q^2)_k^2(q^2;q^4)_k}{(q^2;q^2)_k^2(q^4;q^4)_k}q^{2k}\pmod{[n]\Phi_n(q)^3}.
\end{align*}
By means of Lemma \ref{lem-zero}, \eqref{guo-a} and the last relation, we get the formula: for a positive odd
integer $n\equiv3\pmod{4}$,
\begin{align}\label{guo-conjecture-a}
\sum_{k=0}^{M}(-1)^k[4k+1]\frac{(q;q^2)_k^4(q^2;q^4)_k}{(q^2;q^2)_k^4(q^4;q^4)_k}q^k
&\equiv[n]^2q^{(1-n)/2}\bigg\{1+[n]^2\frac{(n^2-1)(1-q)^2}{24}\bigg\}
\notag\\[5pt]
&\quad\times\frac{(q^3;q^4)_{(n-1)/2}}{(q^5;q^4)_{(n-1)/2}}\pmod{[n]\Phi_n(q)^3}.
\end{align}
It is routine to verify that
\begin{align}\label{guo-conjecture-b}
&[n]^2q^{(1-n)/2}\bigg\{1+[n]^2\frac{(n^2-1)(1-q)^2}{24}\bigg\}\frac{(q^3;q^4)_{(n-1)/2}}{(q^5;q^4)_{(n-1)/2}}
\notag\\[5pt]
&\quad\equiv[n]^2q^{(1+n)/2}\frac{(q^3;q^4)_{(n-1)/2}}{(q^5;q^4)_{(n-1)/2}}\pmod{[n]\Phi_n(q)^3}.
\end{align}
The combination of \eqref{guo-conjecture-a} and
\eqref{guo-conjecture-b} confirms Guo's conjecture
\eqref{guo-conjecture}.

Fixing $c=d=1$ in Theorem \ref{thm-a}, we achieve the $q$-analogue
of \eqref{long-c}.

\begin{cor}\label{corl-a}
Let $n$ be a positive odd integer. Then
\begin{align*}
\sum_{k=0}^{M}[4k+1]\frac{(q;q^2)_k^6}{(q^2;q^2)_k^6}q^k
&\equiv[n]q^{(1-n)/2}\bigg\{1+[n]^2\frac{(n^2-1)(1-q)^2}{24}\bigg\}\\[5pt]
&\quad\times\sum_{k=0}^{(n-1)/2}\frac{(q;q^2)_k^4}{(q^2;q^2)_k^4}q^{2k}\pmod{[n]\Phi_n(q)^3}.
\end{align*}
\end{cor}

Setting $c=d=-1$ in Theorem \ref{thm-a}, we gain the first new
$q$-analogue of \eqref{long-a}.

\begin{cor}\label{corl-e}
Let $n$ be a positive odd integer. Then
\begin{align*}
\sum_{k=0}^{M}[4k+1]\frac{(q;q^2)_k^2(q^2;q^4)_k^2}{(q^2;q^2)_k^2(q^4;q^4)_k^2}q^k
&\equiv[n]q^{(1-n)/2}\bigg\{1+[n]^2\frac{(n^2-1)(1-q)^2}{24}\bigg\}\\[5pt]
&\quad\times\sum_{k=0}^{(n-1)/2}\frac{(q;q^2)_k^4}{(q^4;q^4)_k^2}q^{2k}\pmod{[n]\Phi_n(q)^3}.
\end{align*}
\end{cor}

Letting $c=-1,d\to\infty$ in Theorem \ref{thm-a}, we obtain the
second new $q$-analogue of \eqref{long-a}.

\begin{cor}\label{corl-f}
Let $n$ be a positive odd integer. Then
\begin{align*}
\sum_{k=0}^{M}[4k+1]\frac{(q;q^2)_k^3(q^2;q^4)_k}{(q^2;q^2)_k^3(q^4;q^4)_k}q^{k^2+k}
&\equiv[n]q^{(1-n)/2}\bigg\{1+[n]^2\frac{(n^2-1)(1-q)^2}{24}\bigg\}
\end{align*}
\begin{align*}
&\quad\times\sum_{k=0}^{(n-1)/2}\frac{(q;q^2)_k^3}{(q^2;q^2)_k(q^4;q^4)_k}q^{2k}\pmod{[n]\Phi_n(q)^3}.
\end{align*}
\end{cor}

Letting $c=-1,d\to0$ in Theorem \ref{thm-a}, we get the third new
$q$-analogue of \eqref{long-a}.

\begin{cor}\label{corl-g}
Let $n$ be a positive odd integer. Then
\begin{align*}
\sum_{k=0}^{M}[4k+1]\frac{(q;q^2)_k^3(q^2;q^4)_k}{(q^2;q^2)_k^3(q^4;q^4)_k}q^{-k^2}
&\equiv[n]q^{(1-n)/2}\bigg\{1+[n]^2\frac{(n^2-1)(1-q)^2}{24}\bigg\}\\[5pt]
&\quad\times\sum_{k=0}^{(n-1)/2}\frac{(q;q^2)_k^3}{(q^2;q^2)_k(q^4;q^4)_k}(-q)^{k}\pmod{[n]\Phi_n(q)^3}.
\end{align*}
\end{cor}

Letting $c\to \infty,d\to\infty$ in Theorem \ref{thm-a}, we are led to
the fourth new $q$-analogue of \eqref{long-a}.

\begin{cor}\label{corl-h}
Let $n$ be a positive odd integer. Then
\begin{align*}
\sum_{k=0}^{M}[4k+1]\frac{(q;q^2)_k^4}{(q^2;q^2)_k^4}q^{2k^2+k}
&\equiv[n]q^{(1-n)/2}\bigg\{1+[n]^2\frac{(n^2-1)(1-q)^2}{24}\bigg\}\\[5pt]
&\quad\times\sum_{k=0}^{(n-1)/2}\frac{(q;q^2)_k^3}{(q^2;q^2)_k^2}q^{2k}\pmod{[n]\Phi_n(q)^3}.
\end{align*}
\end{cor}

Letting $c\to0,d\to0$ in Theorem \ref{thm-a}, we arrive at the fifth new
$q$-analogue of \eqref{long-a}.

\begin{cor}\label{corl-i}
Let $n$ be a positive odd integer. Then
\begin{align*}
\sum_{k=0}^{M}[4k+1]\frac{(q;q^2)_k^4}{(q^2;q^2)_k^4}q^{-2k^2-k}
&\equiv[n]q^{(1-n)/2}\bigg\{1+[n]^2\frac{(n^2-1)(1-q)^2}{24}\bigg\}\\[5pt]
&\quad\times\sum_{k=0}^{(n-1)/2}\frac{(q;q^2)_k^3}{(q^2;q^2)_k^2}(-q)^{-k^2}\pmod{[n]\Phi_n(q)^3}.
\end{align*}
\end{cor}

The case $c=d=q^{-2}$ of Theorem \ref{thm-a} yields the following
result.

\begin{cor}\label{corl-m}
Let $n$ be a positive odd integer. Then
\begin{align*}
\sum_{k=0}^{M}[4k+1]\frac{(q;q^2)_k^4(q^{-1};q^2)_k^2}{(q^2;q^2)_k^4(q^4;q^2)_k^2}q^{5k}
&\equiv[n]q^{(1-n)/2}\bigg\{1+[n]^2\frac{(n^2-1)(1-q)^2}{24}\bigg\}\\[5pt]
&\quad\times\sum_{k=0}^{(n-1)/2}\frac{(q;q^2)_k^3(q^5;q^2)_k}{(q^2;q^2)_k^2(q^4;q^2)_k^2}q^{2k}\pmod{[n]\Phi_n(q)^3}.
\end{align*}
\end{cor}

Choosing $M=(p^r-1)/2$ and letting $q\to 1$ in the above
$q$-supercongruence, we obtain
\begin{align*}
\sum_{k=0}^{(p^r-1)/2}(4k+1)\frac{(\frac{1}{2})_k^4(-\frac{1}{2})_k^2}{k!^4(k+1)!^2}
\equiv
p^r\sum_{k=0}^{(p^r-1)/2}\frac{(\frac{1}{2})_k^3(\frac{5}{2})_k}{k!^2(k+1)!^2}\pmod{p^{r+3}}\quad\text{for}\
p>3.
\end{align*}

The case $c=q^{-2},d=1$ of Theorem \ref{thm-a} is the following
result.

\begin{cor}\label{corl-n}
Let $n$ be a positive odd integer. Then
\begin{align*}
\sum_{k=0}^{M}[4k+1]\frac{(q;q^2)_k^5(q^{-1};q^2)_k}{(q^2;q^2)_k^5(q^4;q^2)_k}q^{3k}
&\equiv[n]q^{(1-n)/2}\bigg\{1+[n]^2\frac{(n^2-1)(1-q)^2}{24}\bigg\}\\[5pt]
&\quad\times\sum_{k=0}^{(n-1)/2}\frac{(q;q^2)_k^3(q^3;q^2)_k}{(q^2;q^2)_k^3(q^4;q^2)_k}q^{2k}\pmod{[n]\Phi_n(q)^3}.
\end{align*}
\end{cor}

Fixing $M=(p^r-1)/2$ and letting $q\to 1$ in the above
$q$-supercognruence, we get
\begin{align*}
\sum_{k=0}^{(p^r-1)/2}(4k+1)\frac{(\frac{1}{2})_k^5(-\frac{1}{2})_k}{k!^5(k+1)!}
\equiv
p^r\sum_{k=0}^{(p^r-1)/2}\frac{(\frac{1}{2})_k^3(\frac{3}{2})_k}{k!^3(k+1)!}\pmod{p^{r+3}}\quad\text{for}\
p>3.
\end{align*}

\section{A $q$-supergruence with four parematers modulo\\ $\Phi_n(q)(1-aq^n)(a-q^n)$}
 In this section, we shall deduce a $q$-supercongruence with four parameters modulo
 $\Phi_n(q)(1-aq^n)(a-q^n)$ in terms of the
 creative telescoping method. When the parameters are specified, it
 can produce \eqref{guo-b} and \eqref{guo-d}--\eqref{guo-f}.

 We first require the following lemma (see {\cite[Lemma 3.1]{GS}).

\begin{lem}\label{lem-a}
Let $n$ be a positive odd integer and $k$ an integer with $0\leq
k\leq(n-1)/2$. Then
\begin{align*}
\frac{(xq;q^2)_{(n-1)/2-k}}{(q^2/x;q^2)_{(n-1)/2-k}}\equiv(-x)^{(n-1)/2-2k}\frac{(xq;q^2)_{k}}{(q^2/x;q^2)_{k}}q^{(n-1)^2/4+k}\pmod{\Phi_n(q)}.
\end{align*}
\end{lem}

We also need two more lemmas.
\begin{lem}\label{lem-b}
Let $n$ be a positive odd integer. Then
\begin{align} \label{eq:wei-b}
\sum_{k=0}^{M}[4k+1]\frac{(aq,q/a,q/b,cq,dq,q;q^2)_k}{(q^2/a,aq^2,bq^2,q^2/c,q^2/d,q^2;q^2)_k}\bigg(\frac{bq}{cd}\bigg)^k
\equiv0\pmod{\Phi_n(q)}.
\end{align}
\end{lem}

\begin{proof}
Let $\alpha_q(k)$ denote the $k$-th term on the left-hand side of
\eqref{eq:wei-b}, i.e.,
\begin{align*}
\alpha_q(k)=[4k+1]\frac{(aq,q/a,q/b,cq,dq,q;q^2)_k}{(q^2/a,aq^2,bq^2,q^2/c,q^2/d,q^2;q^2)_k}\bigg(\frac{bq}{cd}\bigg)^k.
\end{align*}
According to Lemma \ref{lem-a}, we obtain
\begin{align*}
\alpha_{q}((n-1)/2-k)\equiv-\alpha_{q}(k)\pmod{\Phi_n(q)}.
\end{align*}
When $(n-1)/2$ is odd, it is not difficult to verify that
\begin{align}\label{equation-a}
\sum_{k=0}^{(n-1)/2}[4k+1]\frac{(aq,q/a,q/b,cq,dq,q;q^2)_k}{(q^2/a,aq^2,bq^2,q^2/c,q^2/d,q^2;q^2)_k}\bigg(\frac{bq}{cd}\bigg)^k
\equiv0\pmod{\Phi_n(q)}.
\end{align}
When $(n-1)/2$ is even, the central term $\alpha_{q}((n-1)/4)$ will
remain. Since it has the factor $[n]$, \eqref{equation-a} is also
true in this instance. If $n>(n-1)/2$, the factor $(1-q^n)$ appears
in the numerator of $\alpha_q(k)$. This indicates
\begin{align*}
\sum_{k=0}^{n-1}[4k+1]\frac{(aq,q/a,q/b,cq,dq,q;q^2)_k}{(q^2/a,aq^2,bq^2,q^2/c,q^2/d,q^2;q^2)_k}\bigg(\frac{bq}{cd}\bigg)^k
\equiv0\pmod{\Phi_n(q)},
\end{align*}
as desired.
\end{proof}

\begin{lem}\label{lem-c}
Let $n$ be a positive odd integer. Then
\begin{align*}
\sum_{k=0}^{M}[4k+1]\frac{(aq,q/a,q/b,cq,dq,q;q^2)_k}{(q^2/a,aq^2,bq^2,q^2/c,q^2/d,q^2;q^2)_k}\bigg(\frac{bq}{cd}\bigg)^k
\equiv0\pmod{[n]}.
\end{align*}
\end{lem}

\begin{proof}
 For $n>1$, let $\zeta\neq1$ be an $n$-th root of unity, which is not
necessarily primitive. This means that $\zeta$ is a primitive root
of unity of odd degree $m|n$. Lemma \ref{lem-b} with $n=m$ implies
that
\begin{align*}
\sum_{k=0}^{m-1}\alpha_{\zeta}(k)=\sum_{k=0}^{(m-1)/2}\alpha_{\zeta}(k)=0.
\end{align*}
By means of the relation:
\begin{align*}
\frac{\alpha_{\zeta}(jm+k)}{\alpha_{\zeta}(jm)}=\lim_{q\to\zeta}\frac{\alpha_{q}(jm+k)}{\alpha_{q}(jm)}=\alpha_{\zeta}(k),
\end{align*}
we achieve
\begin{align*}
\sum_{k=0}^{n-1}\alpha_{\zeta}(k)=\sum_{j=0}^{n/m-1}\sum_{k=0}^{m-1}\alpha_{\zeta}(jm+k)
=\sum_{j=0}^{n/m-1}\alpha_{\zeta}(jm)\sum_{k=0}^{m-1}\alpha_{\zeta}(k)=0,
\end{align*}
\begin{align*}
\sum_{k=0}^{(n-1)/2}\alpha_{\zeta}(k)=
\sum_{j=0}^{(n/m-3)/2}\alpha_{\zeta}(jm)\sum_{k=0}^{m-1}\alpha_{\zeta}(k)+\sum_{k=0}^{(m-1)/2}\alpha_{\zeta}((n-m)/2+k)=0.
\end{align*}
The last two equations tell us that $\sum_{k=0}^{n-1}\alpha_{q}(k)$
and $\sum_{k=0}^{(n-1)/2}\alpha_{q}(k)$ are both divisible by the
cyclotomic polynomials $\Phi_m(q)$. Because this is correct for any
divisor $m>1$ of $n$, we conclude that they are divisible by
\begin{equation*}
\prod_{m|n,m>1}\Phi_m(q)=[n].
\end{equation*}
This completes the proof.
\end{proof}

We now state our main result in this section.
\begin{thm}\label{thm-d}
Let $n$ be a positive odd integer. Then
\begin{align}
&\sum_{k=0}^{M}[4k+1]\frac{(aq,q/a,q/b,cq,dq,q;q^2)_k}{(q^2/a,aq^2,bq^2,q^2/c,q^2/d,q^2;q^2)_k}\bigg(\frac{bq}{cd}\bigg)^k
\notag\\[5pt]\:
&\:\equiv[n](b/q)^{(n-1)/2}\frac{(q^2/b;q^2)_{(n-1)/2}}{(bq^2;q^2)_{(n-1)/2}}
\notag\\[5pt]
&\quad\times\sum_{k=0}^{(n-1)/2}\frac{(aq,q/a,q/b,q/cd;q^2)_k}{(q^2,q^2/b,q^2/c,q^2/d;q^2)_k}q^{2k}\pmod{\Phi_n(q)(1-aq^n)(a-q^n)}.  \label{eq:wei-d}
\end{align}
\end{thm}

\begin{proof}
When $a=q^{-n}$ or $a=q^n$, the left-hand side of  \eqref{eq:wei-d}
is equal to
\begin{align}
&\sum_{k=0}^{M}[4k+1]\frac{(q^{1-n},q^{1+n},q/b,cq,dq,q;q^2)_k}{(q^{2+n},q^{2-n},bq^2,q^2/c,q^2/d,q^2;q^2)_k}\bigg(\frac{bq}{cd}\bigg)^k
\notag\\[5pt]\:
&= {_8\phi_7}\!\left[\begin{array}{cccccccc} q,& q^{\frac{5}{2}},&
-q^{\frac{5}{2}},  & cq,    & dq, & q/b,   & q^{1+n}, & q^{1-n}
\\[5pt]
  & q^{\frac{1}{2}}, & -q^{\frac{1}{2}}, & q^2/c, & q^2/d, & bq^2,  & q^{2-n}, & q^{2+n}
\end{array};q^2,\, \frac{bq}{cd}
\right]. \label{watson-a}
\end{align}
Via \eqref{eq:watson}, the right-hand side of \eqref{watson-a}
can be rewritten as
\begin{align*}
[n](b/q)^{(n-1)/2}\frac{(q^2/b;q^2)_{(n-1)/2}}{(bq^2;q^2)_{(n-1)/2}}
 {_4\phi_3}\!\left[\begin{array}{cccccccc}
q/cd, &q/b, &q^{1+n}, &q^{1-n}
\\[5pt]
  &q^2/c,&q^2/d, &q^2/b
\end{array};q^2,\, q^2\right].
\end{align*}
This proves that the $q$-supercongruence  \eqref{eq:wei-d} holds
modulo $(1-aq^n)$ or $(a-q^n)$. Lemma \ref{lem-b} implies that
\eqref{eq:wei-d} is also true modulo $\Phi_n(q)$. Since $\Phi_n(q)$, $(1-aq^n)$,
and $(a-q^n)$ are pairwise relatively prime polynomials, we obtain
\eqref{eq:wei-d}.
\end{proof}

In the light of Lemma \ref{lem-c}, the $q$-supercongruence \eqref{eq:wei-d} becomes \eqref{guo-b} when $cd=q$. Letting
$a\to1, b=c=d=1$ in \eqref{eq:wei-d} and using \eqref{polynomial} and  Lemma \ref{lem-c}, we get \eqref{guo-d}. If we take
$a\to1,b=c=1$, $d\to \infty$, then the $q$-supercongruence \eqref{eq:wei-d} reduces to
\eqref{guo-e} under \eqref{polynomial} and  Lemma \ref{lem-c}.

Moreover, letting $a\to1$ and putting $b=c=1$, $d=-1$ in \eqref{eq:wei-d} and employing \eqref{polynomial} and  Lemma \ref{lem-c}, we have
\begin{align*}
&\sum_{k=0}^{M}(-1)^k[4k+1]\frac{(q;q^2)_k^4(q^2;q^2)_k}{(q^2;q^2)_k^4(q^4;q^4)_k}q^{k}
\notag\\[5pt]
&\quad\equiv[n]q^{(1-n)/2}\sum_{k=0}^{(n-1)/2}\frac{(q;q^2)_k^2(q^2;q^4)_k}{(q^2;q^2)_k^2(q^4;q^4)_k}q^{2k}\hspace*{-1.5mm}\pmod{[n]\Phi_n(q)^2}.
\end{align*}
In view of the following $q$-supercongruence (see \cite[Theorem
2]{GuoZu2})
\begin{align*}
\sum_{k=0}^{(n-1)/2}\frac{(q;q^2)_k^2(q^2;q^4)_k}{(q^2;q^2)_k^2(q^4;q^4)_k}q^{2k}
\equiv
\begin{cases} \displaystyle \frac{(q^2;q^4)_{(n-1)/4}^2}{(q^4;q^4)_{(n-1)/4}^2}q^{(n-1)/2} \hspace{-3mm} \pmod{\Phi_n(q)^2}, &\text{if $n\equiv 1\pmod 4$,}\\[15pt]
 0  \pmod{\Phi_n(q)^2}, &\text{if $n\equiv 3\pmod 4$,}
\end{cases}
\end{align*}
we immediately get \eqref{guo-f}.

\section{Proof of Theorem \ref{thm-a}}
In this section, we shall establish a $q$-supercongruence with three
parameters modulo $\Phi_n(q)^2(1-aq^n)(a-q^n)$ by using Theorem
\ref{thm-d} and the Chinese remainder theorem for coprime
polynomials. Then we utilize it to give the proof of
Theorem\ref{thm-a}.

In order to reach the goals, we need the following lemma.

\begin{lem}\label{lem-d}
Let $n$ be a positive odd integer. Then
\begin{align}
&\sum_{k=0}^{M}[4k+1]\frac{(aq,q/a,q/b,cq,dq,q;q^2)_k}{(q^2/a,aq^2,bq^2,q^2/c,q^2/d,q^2;q^2)_k}\bigg(\frac{bq}{cd}\bigg)^k
\notag\\[5pt]
&\quad\equiv[n]\frac{(q;q^2)_{(n-1)/2}^2}{(aq^2,q^2/a;q^2)_{(n-1)/2}}  \notag\\[5pt]
&\quad\quad\times\sum_{k=0}^{(n-1)/2}\frac{(aq,q/a,q/b,q/cd;q^2)_k}{(q^2,q^2/b,q^2/c,q^2/d;q^2)_k}q^{2k}\pmod{(b-q^n)}.
\label{eq:lem-b}
\end{align}
\end{lem}

\begin{proof}
When $b=q^{n}$, the left-hand side of \eqref{eq:lem-b} is equal to
\begin{align}
&\sum_{k=0}^{M}[4k+1]\frac{(aq,q/a,q^{1-n},cq,dq,q;q^2)_k}{(q^{2}/a,aq^{2},q^{2+n},q^2/c,q^2/d,q^2;q^2)_k}\bigg(\frac{q^{n+1}}{cd}\bigg)^k
\notag\\[5pt]\:
&= {_8\phi_7}\!\left[\begin{array}{cccccccc} q,& q^{\frac{5}{2}},&
-q^{\frac{5}{2}}, & cq,    & dq,    & aq,    & q/a, & q^{1-n}
\\[5pt]
  & q^{\frac{1}{2}}, & -q^{\frac{1}{2}},  & q^2/c, & q^2/d, & q^2/a, & aq^{2}, & q^{2+n}
\end{array};q^2,\, \frac{q^{n+1}}{cd}
\right]. \label{watson-e}
\end{align}
Via \eqref{eq:watson}, the right-hand side of \eqref{watson-e} can
be rewritten as
\begin{align*}
[n]\frac{(q;q^2)_{(n-1)/2}^2}{(aq^2,q^2/a;q^2)_{(n-1)/2}}
 {_4\phi_3}\!\left[\begin{array}{cccccccc}
q/cd, &aq, &q/a, &q^{1-n}
\\[5pt]
  &q^2/c,&q^2/d, &q^{2-n}
\end{array};q^2,\, q^2\right].
\end{align*}
This proves that Lemma \ref{lem-b} holds modulo $(b-q^n)$ .
\end{proof}

We now give a parametric generaliztation of Theorem \ref{thm-a}.

\begin{thm}\label{thm-g}
Let $n$ be a positive odd integer. Then, modulo
$\Phi_n(q)^2(1-aq^n)(a-q^n)$,
\begin{align*}
&\sum_{k=0}^{M}[4k+1]\frac{(aq,q/a,cq,dq;q^2)_k(q;q^2)_k^2}{(q^2/a,aq^2,q^2/c,q^2/d;q^2)_k(q^2;q^2)_k^2}\bigg(\frac{q}{cd}\bigg)^k
\\[5pt]\:
&\:\equiv[n]\Omega_q(a,n)
\sum_{k=0}^{(n-1)/2}\frac{(aq,q/a,q/cd,q;q^2)_k}{(q^2/c,q^2/d;q^2)_k(q^2;q^2)_k^2}q^{2k},
\end{align*}
where the notation on the right-hand side denotes
\begin{align*}
\Omega_q(a,n)=q^{(1-n)/2}
+q^{(1-n)/2}\frac{(1-aq^n)(a-q^n)}{(1-a)^2}\bigg\{1-\frac{n(1-a)a^{(n-1)/2}}{1-a^n}\bigg\}.
\end{align*}
\end{thm}

\begin{proof}
It is clear that the polynomials $(1-aq^n)(a-q^n)$ and $(b-q^n)$ are
relatively prime. Noting the $q$-congruences
\begin{align*}
&\frac{(b-q^n)(ab-1-a^2+aq^n)}{(a-b)(1-ab)}\equiv1\pmod{(1-aq^n)(a-q^n)},
\\[5pt]
&\qquad\qquad\frac{(1-aq^n)(a-q^n)}{(a-b)(1-ab)}\equiv1\pmod{(b-q^n)}
\end{align*}
and employing the Chinese remainder theorem for coprime polynomials,
we can derive, from Lemmas \ref{lem-b} and \ref{lem-d} and Theorem
\ref{thm-d}, the following $q$-congruence: modulo
$\Phi_n(q)(1-aq^n)(a-q^n)(b-q^n)$,
\begin{align}
&\sum_{k=0}^{M}[4k+1]\frac{(aq,q/a,q/b,cq,dq,q;q^2)_k}{(q^2/a,aq^2,bq^2,q^2/c,q^2/d,q^2;q^2)_k}\bigg(\frac{bq}{cd}\bigg)^k
\notag\\[5pt]\:
&\quad\equiv[n]R_q(a,b,n)
\sum_{k=0}^{(n-1)/2}\frac{(aq,q/a,q/b,q/cd;q^2)_k}{(q^2,q^2/b,q^2/c,q^2/d;q^2)_k}q^{2k},
\label{eq:par}
\end{align}
where
\begin{align*}
R_q(a,b,n)&=\frac{(b-q^n)(ab-1-a^2+aq^n)}{(a-b)(1-ab)}\frac{(b/q)^{(n-1)/2}(q^2/b;q^2)_{(n-1)/2}}{(bq^2;q^2)_{(n-1)/2}}
\\[5pt]
&\quad+\frac{(1-aq^n)(a-q^n)}{(a-b)(1-ab)}\frac{(q;q^2)_{(n-1)/2}^2}{(aq^2,q^2/a;q^2)_{(n-1)/2}}.
\end{align*}

Letting $b\to 1$ in \eqref{eq:par}, we conclude that, modulo $\Phi_n(q)^2(1-aq^n)(a-q^n)$,
\begin{align}
&\sum_{k=0}^{M}[4k+1]\frac{(aq,q/a,cq,dq;q^2)_k(q;q^2)_k^2}{(q^2/a,aq^2,q^2/c,q^2/d;q^2)_k(q^2;q^2)_k^2}\bigg(\frac{q}{cd}\bigg)^k
\notag\\[5pt]\:
&\:\equiv[n]S_q(a,n)
\sum_{k=0}^{(n-1)/2}\frac{(aq,q/a,q/cd,q;q^2)_k}{(q^2/c,q^2/d;q^2)_k(q^2;q^2)_k^2}q^{2k},
\label{formula-a}
\end{align}
where
\begin{align*}
S_q(a,n)=\frac{(1-q^n)(1+a^2-a-aq^n)}{(1-a)^2}q^{(1-n)/2}
-\frac{(1-aq^n)(a-q^n)}{(1-a)^2}\frac{(q;q^2)_{(n-1)/2}^2}{(aq^2,q^2/a;q^2)_{(n-1)/2}}.
\end{align*}
Two formulas due to Guo \cite[Lemma 2.1]{Guopreprint} can be stated
as
\begin{align*}
&(aq^2,q^2/a;q^2)_{(n-1)/2}
\equiv(-1)^{(n-1)/2}\frac{(1-a^n)q^{-(n-1)^2/4}}{(1-a)a^{(n-1)/2}}\pmod{\Phi_n(q)},
\\[5pt]
&(aq,q/a;q^2)_{(n-1)/2}
\equiv(-1)^{(n-1)/2}\frac{(1-a^n)q^{(1-n^2)/4}}{(1-a)a^{(n-1)/2}}\pmod{\Phi_n(q)},
\end{align*}
from which we deduce that
\begin{align*}
\frac{(q;q^2)_{(n-1)/2}^2}{(aq^2,q^2/a;q^2)_{(n-1)/2}}\equiv\frac{n(1-a)a^{(n-1)/2}}{(1-a^n)q^{(n-1)/2}}\pmod{\Phi_n(q)},
\end{align*}
and so
\begin{align}
[n]S_q(a,n)\equiv[n]\Omega_q(a,n)\pmod{\Phi_n(q)^2(1-aq^n)(a-q^n)}.\label{formula-b}
\end{align}
Using \eqref{formula-a} and \eqref{formula-b}, we gain the desired result.
\end{proof}

\begin{proof}[Proof of Theorem \ref{thm-a}]
By L'H\^{o}spital's rule, we have
\begin{align*}
&\lim_{a\to1}\frac{(1-aq^n)(a-q^n)}{(1-a)^2}\bigg\{1-\frac{n(1-a)a^{(n-1)/2}}{1-a^n}\bigg\}\\[5pt]
&=\lim_{a\to1}\frac{(1-aq^n)(a-q^n)\{1-a^n-n(1-a)a^{(n-1)/2}\}}{(1-a)^2(1-a^n)}\\[5pt]
&=[n]^2\frac{(n^2-1)(1-q)^2}{24}.
\end{align*}
Letting $a\to1$ in Theorem \ref{thm-g} and utilizing the above
limit, we arrive at \eqref{eq:wei} through \eqref{polynomial} and  Lemma \ref{lem-c}.
\end{proof}

On the basis of numerical calculations, we would like to put forward
the following conjecture.
\begin{conj}\label{conj}
Let $n$ be a positive integer with $n\equiv 3\pmod{4}$. Then
\begin{align*}
\sum_{k=0}^{M} (-1)^k[4k+1]\frac{(q;q^2)_k^4(q^2;q^{4})_k}
{(q^2;q^2)_k^4(q^4;q^4)_k}q^k \equiv
[n]^2\frac{(q^3;q^4)_{(n-1)/2}}{(q^5;q^4)_{(n-1)/2}}q^{(1-n)/2}
\pmod{[n]\Phi_n(q)^4},
\end{align*}
where $M=(n-1)/2$ or $n-1$. In particular, for $p\equiv 3\pmod{4}$,
\begin{equation*}
\sum_{k=0}^{(p^r-1)/2}(-1)^k(4k+1)\frac{(\frac{1}{2})_k^5}{k!^5}\equiv
p^{2r}\frac{(\frac{3}{4})_{(p^r-1)/2}}{(\frac{5}{4})_{(p^r-1)/2}}\pmod{p^{r+4}}.
\end{equation*}
\end{conj}


\end{document}